\documentclass[12]{amsart}
\usepackage{graphicx}
\usepackage{amscd}
\usepackage{lineno,hyperref}
\usepackage[bottom]{footmisc}
\usepackage{amsmath}
\usepackage{amssymb}
\usepackage{esint}
\usepackage{curves}
\usepackage{array,multirow,graphicx}
\usepackage{tikz}
\usepackage{nicefrac}
\usetikzlibrary{shapes.geometric}
\usetikzlibrary{shapes.multipart}
\usepackage{enumerate}
\usetikzlibrary{positioning}

\newtheorem{theorem}{Theorem}

\newtheorem{corollary}{Corollary}

\newtheorem{lemma}{Lemma}

\newtheorem{remark}{Remark}

\newtheorem{question}{Question}
\numberwithin{equation}{section}

\begin{document}
\markboth{Authors' Names}
{Classification of Groups according to the number of end vertices}
\title[]{Classification of Groups according to the number of end vertices in the coprime graph}
\maketitle
\vspace{10pt}
 \begin{center} \footnotesize TARIQ A. ALRAQAD \\
Department of Mathematics, University of Hail\\ Hail, Kingdom of Saudi Arabia \\
t.alraqad@uoh.edu.sa\\
\vspace{5pt}
MUHAMMAD  S. SAEED\\
Department of Mathematics, University of Hail\\ Kingdom of Saudi Arabia\\
ms.saeed@uoh.edu.sa\\
\vspace{5pt}
ETAF S. ALSHAWARBEH\\
Department of Mathematics, University of Hail\\ Kingdom of Saudi Arabia\\
e.alshawarbeh@uoh.edu.sa\\
\end{center}

\begin{abstract}
In this paper we characterize groups according to the number of end vertices in the associated coprime graphs. An upper bound on the order of the group that depends on the number of end vertices is obtained. We also prove that $2-$groups are the only groups whose coprime graphs have odd number of end vertices. Classifications of groups with small number of end vertices in the coprime graphs are given. One of the results shows that $\mathbb{Z}_4$ and $\mathbb{Z}_2\times \mathbb{Z}_2$ are the only groups whose coprime graph has exactly three end vertices.   
\end{abstract}



\section{Introduction}

Throughout this paper, $G$ denotes a finite group. The order of a group $G$ is denoted by $|G|$ and the order of an element $g\in G$ is denoted by 𝑜$|g|$. For a general reference on group theory and graph theory, we refer the reader to the text-books of \cite{Ro02} and \cite{Bo76} respectively. The greatest common divisor of two integers $a$ and $b$ is denoted by $(a,b)$. The radical of an integer $n\geq 2$, denoted by $rad(n)$, is the product of the distinct prime divisors of $n$, and $rad(1)=1$.  

The coprime graph of a finite group $G$ is the simple graph $\Gamma_G$ whose vertices are the elements of $G$ with two vertices $x$ and $y$ are adjacent if $(|x|,|y|)=1$. It is clear that the identity element of $G$ (usually denoted by $e$) is adjacent with every other element; so the graph is connected with diameter at most $2$. The concept of coprime graphs was introduced by Sattanathan and  Kala \cite{Sa09} where they call them order prime graphs of finite groups. Later Ma et al \cite{Ma14} studied these graphs and called them coprime graphs of groups.  Further studies of these graphs were done in \cite{Do16,Ra16,Se17}.

The degree of a vertex in a graph is  the number of vertices in the graph that are adjacent to it. A vertex that has degree one is called an end vertex. For a group $G$, we denot by $E_G$ the set of all end vertices in $\Gamma_G$ other than the identity element. Note that $E_G$ may be empty. In \cite[section 3]{Ma14},  Ma et al  characterized the groups whose coprime graphs have exactly $0$, $1$, or $2$ end vertices, that is $|E_G|=0,1,2$. Then they ended the section with the following question:   

\begin{question}\cite[Question 3.7]{Ma14}
Is it possible to characterize all finite groups $G$ whose coprime graph contains precisely three end vertices?
\end{question}     

In \cite{Do16}, Dorbidi has proved a result that gives the number of end vertices in a coprime graph of any nilpotent group. 

In this paper our gaol is to classify groups according to the number of the end vertices in the associated coprime graphs. In Section (\ref{mr}), we present the main results. We prove that one of the end vertices of $\Gamma_G$ has order equal to a power of a prime $p$ if and only if $G$ is a $p-$group. We conclude from this result that $|E_G|$ is odd if and only if $G$ is a $2-$group. In addition, we obtain an upper bound on the order of the group that depends on the size of $E_G$.  Section (\ref{cl}) is devoted to classifications of groups whose coprime graphs have small number of end vertices mainly between $1$ and $10$. Corollary (\ref{e3}) states that $\mathbb{Z}_4$ and $\mathbb{Z}_2\times \mathbb{Z}_2$ are the only groups whose coprime graph has exactly three end vertices. This gives a complete answer to \cite[Question 3.7]{Ma14}. 

\section{Main Result}\label{mr}

We start this section by the following result, which characterize the coprime graphs of $p-$groups. For proof see \cite[Theorem 2.7]{Sa09} or \cite[Proposition 2.6]{Ma14} 

\begin{theorem}\label{pgs}
Let $G$ be a group. Then $\Gamma_G\cong K_{1,|G|-1}$ (star graph) if and only if $G$ is a $p-$group with some prime integer $p$. 
\end{theorem}

Now we prove the following useful lemma.

\begin{lemma}\label{rad}
Let $G$ be a group with $|G|\geq3$. Then
\begin{enumerate}[(i)]
\item $E_G=\{x\in G \mid rad(|x|)=rad(|G|)\}$,
\item for each $x\in E_G$, $\phi(\left|x\right|)\leq |E_G|$, ($\phi$ denotes the Euler's phi function) and
\item if there exists $x\in E_G$ such that $\phi(\left|x\right|)=|E_G|$, then $rad(|x|)=|x|$ and $\left\langle x\right\rangle$ is a unique cyclic subgroup of $G$ of order $|x|$.
\end{enumerate}
\end{lemma}

\begin{proof}
(i) By Cauchy Theorem  we note that for each prime number $p$ dividing $|G|$ there is an element of $G$ whose order is $p$. So an element $x\in G$ is an end vertex of $\Gamma_G$ if and only if every prime divisor of $|G|$ divides $|x|$, which is equivalent to say that $rad(|x|)=rad(|G|)$. 

(ii) Let $x\in E_G$. Then any generator of $\left\langle x\right\rangle$ is in $E_G$. Since $\left\langle x\right\rangle$ has exactly $\phi(|x|)$ generators, we must have $\phi(|x|)\leq |E_G|$.

(iii) Suppose that  $\phi(\left|x\right|)=|E_G|$. Then $E_G$ consists of the generators of $\left\langle x\right\rangle$, that is $E_G=\{x^j \mid (j,|x|)=1\}$. Assume that $rad(|x|)< |x|$ and let $d=\frac{|x|}{rad(|x|)}$. Then  $|x^d|=\frac{|x|}{d}=rad(|x|)$. So $rad(|x^d|)=rad(|x|)$, which yields $x^d\in E_G$. But $x^d$ is not a generator of  $\left\langle x\right\rangle$, a contradiction. So $rad(|x|)=|x|$. Moreover, if $|y|=|x|$ for some $y\in G$, then $y\in E_G$. So $\left\langle y\right\rangle=\left\langle x\right\rangle$. This completes the proof.
\end{proof}

\begin{theorem}\label{epg}
Let $G$ be a finite group. Then $E_G$ contains a non-identity element whose order is a power of prime $p$ if and only if $G$ is a $p-$group. Moreover, $|G|=|E_G|+1$.
\end{theorem}

\begin{proof}
Let $x$ be an end vertex in $\Gamma_G$ such that $|x|=p^k$ for some positive integer $k$. Then by Lemma (\ref{rad}), $p=rad(|x|)=rad(|G|)$, and hence $G$ is a $p-$group. Moreover, Theorem (\ref{pgs}) implies that $\Gamma_G$ is a star graph, which yields every non-identity element of $G$ is an end vertex. Hence $|G|=|E_G|+1$.  The converse follows directly by applying Theorem (\ref{pgs}).
\end{proof}

The following result gives a characterization of all finite groups whose coprime graphs have odd number of end vertices.

\begin{theorem}\label{oev}
Let $G$ be a finite group. Then $|E_G|$ is odd if and only if $G$ is a $2-$group. Moreover, $|E_G|=2^n-1$ for some positive integer $n$.
\end{theorem}
\begin{proof}
Suppose $|E_G|$ is odd. Since $|x|=|x^{-1}|$ for every $x\in G$,  $x\in E_G$ if and only if $x^{-1}\in E_G$. Furthermore, $|E_G|$ is odd, so there must be $x\in E_G$ such that $x=x^{-1}$. This implies that $E_G$ contains an element of order $2$. Thus by Theorem (\ref{epg}) we conclude that $G$ is a $2-$group, and $|E_G|=|G|-1$.

Now, if $G$ is a $2-$group, then $|E_G|=|G|-1$ which is odd. This proves the converse.
\end{proof}

In order to classify all groups (up to isomorphism) whose coprime graphs have the same number of end vertices, a natural question arises is that how many possible orders of such groups are there? In the case $|E_G|=0$, one can find that there are infinitely many groups whose coprime graphs have no end vertices. For instance, if $q\geq 5$ is a prime then $E_{S_3\times \mathbb{Z}_q}$ is empty. As a matter of fact if $G$ is a nonabelian finite group and $rad(|G|)=|G|$ then $E_G$ is empty. For the remainder of this section the goal is to establish an upper bound on the order of the group $G$ depending on the size of $E_G$. We will only be concerned with the cases when $|E_G|$ is even because the cases when $|E_G|$ is odd are completely solved in Theorem (\ref{oev}). 

The centralizer of an element $x$ in a group $G$ is defined to be $C_G(x)=\{g\in G \mid gx=xg\}$, and the conjugacy class of $G$ containing $x$ is $Cl_G(x)=\{gxg^{-1} \mid g\in G\}$. It is well known in group theory that $C_G(x)$ is a subgroup of $G$ and its index equal to the size of $Cl_G(x)$; that is $[G:C_G(x)]=|Cl_G(x)|$. The following theorem gives a characterization of the centralizers of the end vertices. 

\begin{theorem}\label{cen}
Let $G$ be a group and let $x\in E_G$. Then 
\[C_G(x) = \bigcup_{_{y\in E_G\cap C_G(x)}}\left\langle y\right\rangle.\]
\end{theorem}
\begin{proof}
Let $H= \bigcup_{y\in E_G\cap C_G(x)}\left\langle y\right\rangle$. If $y\in  E_G\cap C_G(x)$ then $y\in C_G(x)$, and hence 
$\left\langle y \right\rangle\subseteq C_G(x)$. So $H\subseteq C_G(x)$. Now we need to show that $C_G(x) \subseteq H$.  Clearly $E_G\cap C_G(x)\subseteq H$. Now, let $g\in C_G(x)$. If $rad(|g|)=rad(|G|)$ then $g\in E_G \cap C_G(x) \subseteq F$. Assume $rad(|g|)< rad(|G|)$, then we need to show that $g\in \left\langle y\right\rangle$ for some $y\in E_G\cap C_G(x)$. Let $d=rad(|G|)/rad(|g|)$ and let $z=x^{|x|/d}$. Then we have $(|z|,|g|)=(d,|g|)=1$ and $zg=gz$. So $|gz|=|g||z|$, which implies $rad(|gz|)=rad(|g|)rad(|z|)=rad(|G|)$. So $gz\in E_G\cap C_G(x)$, which yields $\left\langle gz \right\rangle \subseteq F$. Now,  
\[(gz)^d=\left(gx^{|x|/d}\right)^d=g^dx^{|x|}=g^d.\]
Also since $(|g|,d)=1$,  we have $\left\langle g\right\rangle=\left\langle g^d\right\rangle$. So $\left\langle g\right\rangle=\left\langle (gz)^d\right\rangle$, and hence $g\in \left\langle (gz)^d\right\rangle \subseteq F$. Therefore $C_G(x) \subseteq F$.
\end{proof}

The next theorem gives an upper bound on the order of $G$ depending on $|E_G|$.

\begin{theorem}\label{lb}
Let $n$ be a positive integer and $M$ be the largest possible integer such that $\phi(M)\leq 2n$. If $|E_G|=2n$ then 
\[|G| \leq 2n(Mn-n+1).\].
\end{theorem}
\begin{proof}
Assume that $|E_G|=2n$ and let $x\in E_G$. Then $Cl_G(x)\subseteq E_G$, and so $|Cl_G(x)|\leq |E_G|=2n$. 
Let $m=\displaystyle\min_{y\in E_G} \phi(|y|)$. Then for each $y\in E_G\cap C_G(x)$, the size of the set $\{z\in E_G\cap C_G(x)| \left\langle z\right\rangle=\left\langle y\right\rangle\}$ is at least $m$. Also we know that $e\in \left\langle y\right\rangle$ for all $y$. Theorem (\ref{cen}) implies that 
\begin{equation} \label{eq:bd} |C_G(x)| = \left|\bigcup_{_{y\in E_G\cap C_G(x)}}\left\langle y\right\rangle\right| \leq 1+\frac{1}{m}\sum_{_{y\in E_G\cap C_G(x)}}(|y|-1).\end{equation}

Now, from Lemma (\ref{rad}) it is clear that $|y|\leq M$ for all $y\in E_G$. Substituting this in \ref{eq:bd} yields 

\[|C_G(x)|\leq 1+\frac{(M-1)|E_G\cap C_G(x)|}{m}.\] Since no element in $E_G$ has order $2$, we obtain that $m\geq 2$.  Thus  

\[|C_G(x)|\leq 1+\frac{(M-1)|E_G\cap C_G(x)|}{2}\leq 1+\frac{(M-1)|E_G|}{2}=1+\frac{(M-1)2n}{2}=Mn-n+1.\]

\[|G|=[G:C_G(x)]|C_G(x)|=|Cl_G(x)||C_G(x)|\leq 2n(Mn-n+1).\]

\end{proof}

Under some conditions the bound given in Theorem (\ref{lb}) can be improved. The following theorem gives a better bound on the size of the group under the condition that all end vertices are generated by the same element. We find this results is very useful in the next section.
\begin{theorem}\label{bound}
Let $G$ be a group with $|G|\geq3$ and $|E_G|\geq1$. If there exists $x \in G$ such that $E_G\subseteq \left\langle x\right\rangle$, then $|G|\leq |x|\phi(|x|)$.
\end{theorem}
\begin{proof}
Assume that   $E_G\subseteq \left\langle x\right\rangle$ for some $x\in G$ and let $x^t\in E_G$. So $rad(|x^t|)=rad(|G|)$. Since $rad(|x^t|)$ divides $rad(|x|)$ which divides $rad(|G|)$, we obtian that $rad(|x|)=rad(|G|)$, which yields $x\in E_G$. This implies that $Cl_G(x)\subseteq E_G\subseteq\left\langle x\right\rangle$. Thus $|Cl_G(x)|\leq \phi(|x|)$. Also Theorem (\ref{cen}) implies that $C_G(x)\subseteq \left\langle x\right\rangle$. Thus we have  \[[G:\left\langle x\right\rangle]=\left[G:C_G\left(x\right)\right]=\left|Cl\left(x\right)\right|\leq \phi(|x|).\]

Hence $|G|=[G:\left\langle x\right\rangle]|x|\leq |x|\phi(|x|)$.
\end{proof}

\begin{remark}\label{sharp} Taking $|E_G|=n$ in Theorem (\ref{bound}) yields $|G|\leq Mn$.  This bound is met by some groups as we will see later in Theorem (\ref{e2}). 
\end{remark}

\section{Classifications of groups with small number of end vertices}\label{cl}

In this section we use the theorems presented in Section (\ref{mr}) to classify groups whose coprime graphs have small number of end vertices. We start with the odd cases because they follow directly from Theorem (\ref{oev}). It is shown in \cite{Ma14} that $|E_G|=1$ if and only if $G\cong \mathbb{Z}_2$. We list this result here for completeness and we give a shorter proof for it.  

\begin{corollary}\label{e1} Let $G$ be a finite group. Then $|E_G|=1$ if and only if $G\cong \mathbb{Z}_2$.
\end{corollary}
\begin{proof} $|E_G|=1\Leftrightarrow |G|=2 \Leftrightarrow G\cong \mathbb{Z}_2$. \end{proof}

The following result gives a complete answer to question 3.7. in \cite{Ma14}.

\begin{corollary}\label{e3}
For a finite group $G$, $|E_G|=3$ if and only if $G\cong \mathbb{Z}_4$ or $G\cong \mathbb{Z}_2\times \mathbb{Z}_2$.
\end{corollary}
\begin{proof}
$|E_G|=3\Leftrightarrow |G|=4 \Leftrightarrow G\cong \mathbb{Z}_4 \text { or }G\cong \mathbb{Z}_2\times \mathbb{Z}_2$.
\end{proof}

Theorem (\ref{oev}) states that the only possible odd values for $|E_G|$ are of the form $2^n-1$. So there are no groups $G$ for which  
$E_G$ consists of five elements, nine elements, eleven elements, ... etc. Using similar arguments as in Corollaries (\ref{e1}) and (\ref{e3}) it can be shown that  $|E_G|=7$ if and only if $|G|=8$, which means $G$ could be any one of the five groups of order eight.  

For an integers $k\geq 2$, $D_{2k}$ denote the dihedral group of order $2k$, and $Dic_{4k}$ denote the dicyclic group of order $4k$ which has presentation $\langle a,b \mid a^k=b^2=e, bab^{-1}=a^{-1}\rangle$. For a prime power $q$, $GA(1,q)$ denotes the general affine group of degree $1$ over the field of $q$ elements. Next we deal with some cases in which $|E_G|$ is even. The following theorem classifies all groups whose coprime graphs have two end vertices, and this is a complete improvement of \cite[Theorem 3.5]{Ma14}.

\begin{theorem}\label{e2}
For a finite group $G$, $|E_G|=2$ if and only if $G\cong \mathbb{Z}_3, \mathbb{Z}_6, D_{12}, Dic_{12}$ 
\end{theorem}
\begin{proof}
Suppose $|E_G|=2$. If there is $x\in E_G$ such that $rad(|x|)=p$, where $p$ is a prime, then $G$ is a $p-$group and $|G|=|E_G|+1=3$. Thus $G\cong \mathbb{Z}_3$.

Now assume that $rad(|x|)$ is not a prime for all $x\in E_G$. Form Lemma (\ref{rad}), we have that $\phi(|x|)\leq |E_G|=2$ for all $x\in E_G$, which yields $|x|=6$ for all $x\in E_G$. Since an element and its inverse have the same order, there exists $x\in G$ such that $|x|=6$ and $E_G=\{x,x^5\}\subseteq \left\langle x\right\rangle$. Thus Theorem (\ref{bound}) implies that $|G|\leq |x|\phi(|x|)=12$. Furthermore,  $rad(|G|)=rad(|x|)=6$. So  $|G|=6$ or $|G|=12$. If $|G|=6$, then $G=\left\langle x\right\rangle\cong \mathbb{Z}_6$. Also, from the groups of order $12$ only $D_{12}$ and $Dic_{12}$ has exactly two end vertices. Thus $G\cong \mathbb{Z}_6, D_{12}, Dic_{12}$. 
\end{proof}

\begin{remark} Each one of the groups $D_{12}$ and $Dic_{12}$ meet the bound given in Remark (\ref{sharp}). Each group contains element of order $6$ that generate both end vertices and $2\times 6=12$ which is the order of each group.
\end{remark}

\begin{theorem}\label{e4}
For a finite group $G$, $|E_G|=4$ if and only if $G\cong \mathbb{Z}_5$, $\mathbb{Z}_{10}$, $D_{20}$, $Dic_{20}$ ,$GA(1,5)\times \mathbb{Z}_2$, or $\langle a,b \mid a^5 = b^8 = e, bab^{-1} = a^2 \rangle.$
\end{theorem}
\begin{proof}
Suppose $|E_G|=4$. If there is $x\in E_G$ such that $rad(|x|)=p$, where $p$ is a prime, then $G$ is a $p-$group and $|G|=|E_G|+1=5$. So $G\cong \mathbb{Z}_5$.

Now assume that $rad(|x|)$ is not a prime for all $x\in E_G$. By Lemma (\ref{rad}), we have $\phi(|x|)\leq 4$. So the possible orders of the elements in $E_G$ are $6$, $10$, and $12$. If  $|x|=12$ for some $x\in E_G$, then $|x|=|x^5|=|x^7|=|x^{11}|=12$ and $|x^2|=|x^{10}|=6$. This implies that $|E_G|\geq6$, a contradiction. 

Now Assume that $|x|=10$ for some $x\in E_G$. Then $rad(|G|)=10$ and $E_G=\{x, x^3, x^7, x^9\}\subseteq \left\langle x \right\rangle$. Now Theorem \ref{bound} implies that $|G|\leq |x|\phi(|x|)=40$. So $|G|=10, 20, \text{ or } 40$. After examining all groups of orders $10$, $20$, and $40$ we conclude that $G\cong \mathbb{Z}_{10}, D_{20}, Dic_{20}, GA(1,5)\times \mathbb{Z}_2, \text{ or } \langle a,b \mid a^5 = b^8 = e, bab^{-1} = a^2 \rangle$.

The remaining case is that $|x|=6$ for all $x\in E_G$. In this case $E_G=\{x , x^5, y, y^5\}$ where $x,y\in G$ with $x\neq y$, $x\neq y^{-1}$, and $|x|=|y|=6$. Now $yxy^{-1}\in E_G$ but also $yxy^{-1}\neq y \text{ and } yxy^{-1}\neq y^{-1}$. So either $yxy^{-1}=x$ or $yxy^{-1}=x^5$.

Case 1: $yxy^{-1}=x$. Then Theorem (\ref{cen}) implies that $C_G(x)=\left\langle x\right\rangle\cup\left\langle y\right\rangle$. So $|C_G(x)|\leq 11$, which yields $|C_G(x)|=6$. So $C_G(x) =\left\langle x\right\rangle$, which implies $y\in \left\langle x\right\rangle$, a contradiction because $y\neq x$ and $y \neq x^{-1}$.

Case 2: $yxy^{-1}=x^{-1}$. So $y,y^{-1} \notin C_G(x)$. Then again Theorem (\ref{cen}) implies $C_G(x)=\left\langle x\right\rangle$.
Thus we have \[|G|=[G:C_G(x)]|C_G(x)|=|Cl_G(x)||x|\leq 24\] So $|G|=6, 12, 18, \text{ or } 24$. Examining all groups of these orders shows that none of them has exactly four elements of order $6$.
\end{proof}

\begin{theorem}\label{e6}
For a finite group $G$, $|E_G|=6$ if and only if $G$ is isomorphic to one of the following groups:
\begin{enumerate}[(i)]
\item (Groups of order $7$) $\mathbb{Z}_{7}$. 
\item (Groups of order $12$) $\mathbb{Z}_{12}$, $\mathbb{Z}_2\times \mathbb{Z}_6$. 
\item (Groups of order $14$) $\mathbb{Z}_{14}$. 
\item (Groups of order $18$) $S_3\times \mathbb{Z}_3$.
\item (Groups of order $24$) $D_{24}$, $Dic_{24}$, $S_3\times \mathbb{Z}_4$, $Dic_{12}\times \mathbb{Z}_2$, $D_{12}\times \mathbb{Z}_2$, 
$\langle a,b \mid a^3 = b^8 = e, bab^{-1} = a^{-1} \rangle$, $\langle a,b,c \mid a^2=b^2=c^3=(ac)^2=(ba)^4=e, bc=cb\rangle $.
\item (Groups of order $28$) $D_{28}$, $Dic_{28}$.
\item (Groups of order $36$) $C_3\times A_4$, $\left\langle a,b \mid a^9=b^2=e, (a^{-1}b)^2=ba^{-2} \right\rangle$.
\item (Groups of order $72$)\\ $\langle a,b,c\mid a^2=b^2=c^9=(ac)^2=e, bac^{-1}=cab,  abcb=bac, bc^3=c^3b\rangle$, \\
$\langle a,b,c,d \mid a^2=b^2=c^3=d^3=(bc)^3=(abc)^2=e, ad=d^2a, ac=c^2a, cd=dc, bd=db\rangle.$
\end{enumerate}
\end{theorem}
\begin{proof}
Suppose $|E_G|=6$. Assume $E_G$ contains an element whose order is a power of a prime $p$, then $G$ is a $p-$ group and $|G|=|E_G|+1=7$. So $G\cong \mathbb{Z}_7$.

Now assume the order of any element of $E_G$ is not a power of a prime. By Lemma (\ref{rad}), we have $\phi(|x|)\leq 6$. So the possible orders of the elements in $E_G$ are $6$, $10$, $12$, $14$, and $18$. If  $|x|=18$ for some $x\in E_G$, then $|x|=|x^5|=|x^7|=|x^{11}|=|X^{13}|=|X^{17}|=18$ and $|x^3|=|x^{15}|=6$. This implies that $|E_G|\geq8$, a contradiction.

Case 1: $|x|=14$ for some $x\in E_G$. Then $rad(|G|)=14$ and \[E_G=\{x, x^3, x^5, x^9, x^{11}, x^{13}\} \subseteq \left\langle x \right\rangle.\] Now Theorem (\ref{bound}) implies that $|G|\leq |x|\phi(|x|)=84$. So $|G|=14, 28, \text{ or } 56$.  

Case 2: $|x|=12$ for some $x\in E_G$. Then $rad(|G|)=6$ and \[E_G=\{x, x^2, x^5, x^7, x^{10}, x^{11}\}\subseteq \left\langle x \right\rangle.\]Again Theorem \ref{bound} implies that $|G|\leq |x|\phi(|x|)=48$. So $|G|=12, 24, 36, \text{ or } 48$. 

Case 3: $|x|=10$ for some $E_G$. Then the set $U=\{x, x^3,  x^7, x^9\}\subseteq E_G\}$. So there is $y\in (E_G\setminus U)$ such that $\phi(|y|)=2$ and $rad(|y|)=10$, which is impossible.   

Case 4: $|x|=6$ for all $x\in E_G$. In this case $E_G=\{x , x^5, y, y^5, z, z^5\}$, where $x$, $y$, and $z$ are elements of order $6$ order in $G$. Also we have $rad(|G|=6)$ and $|Cl_G(x)|\leq 6$. Furthermore, Theorem  (\ref{cen}) implies that $|C_G(x)|\leq 16$. Hence $|G|=6,12,18,24,36,48,54,72, 96.$ 

From all cases combined we conclude that the possible orders for $G$ are $6$, $14$, $12$, $18$, $24$, $28$, $36$, $48$, $54$, $56$, $72$, $96$. After analyzing the all coprime graphs of groups having these orders we obtain the result in theorem. \end{proof}

For any positive integer $n$, the groups $G$, for which $|E_G|=n$, can be determined using the same technique presented in this section. We classified all groups $G$ for which $|E_G|=8$ or $|E_G|=10$. We mention these results for completeness however the proofs are omitted. Regarding the case $|E_G|=8$, there are  $48$ groups $G$ for which $|E_G|=8$. We list below the GAP ID's \cite{GAP17} for these groups. These groups can be obtained from the GAP's SmallGroup library using the command "$SmallGroup(n,k)$" , where $(n,k)$ is the ID of the group. 

\noindent $(9,1),$ $(9,2),$ $(15,1),$ $(18,2),$ $(18,5),$ $(24,3),$ $( 24,13),$  $( 30,4),$  $(36,1),$  $(36,4),$  $(36,7),$  $(36,13),$  $(48,28),$  $(48,29),$  $(48,30),$ $(48,48),$  $(60,1),$ $(60,2),$  $(60,3),$  $(60,10),$ $(60,11),$  $(60,12),$  $(72,19),$  $(72,45),$  $(120,6),$  $(120,7),$  $(120,8),$  $(120,9),$  $(120,10),$  $(120,11),$  $(120,12),$  $(120,13),$  $(120,14),$  $(120,40),$  $(120,41),$  $(120,42),$  $(144,114),$  $(144,120),$  $(144,185),$  $(144,187),$  $(240,95),$  $(240,96),$  $(240,97),$  $(240,98),$  $(240,99),$  $(240,100),$  $(240,101),$ $(240,195)$.    

For the case $|E_G|=10$, we find that  the groups $\mathbb{Z}_{11}, \mathbb{Z}_{22}, D_{44}$, and $Dic_{44}$ are the only groups whose coprime graph has exactly $10$ end vertices.

\end{document}